\documentclass[a4,11pt]{paper}
\usepackage{amsthm,amstext,amsfonts,bm,amssymb,amsmath}
\usepackage{enumitem,color,stmaryrd,tikz-cd,mathtools,mathrsfs}
\usepackage[english]{babel}  
\usepackage{graphicx}
\usepackage{lipsum}
\usepackage{xcolor}
\usepackage{booktabs}
\usepackage{scalerel}
\usepackage{bbold}
\usepackage{dutchcal}
\usepackage[top=0.9in, left=0.9in, right=0.9in, bottom=0.9in]{geometry}

\newcommand{\Z}{\mathbb{Z}}

\newcommand{\Eps}{\mathsf{E}}

\newcommand\numberthis{\addtocounter{equation}{1}\tag{\theequation}}
\subsectionfont{\large\sf\bfseries\color{black!50!blue}} 
\sectionfont{\Large\sf\bfseries\color{black!50!blue}} 
\titlefont{\LARGE\sf\bfseries\color{black!50!blue}} 
\numberwithin{equation}{section}

\newcommand{\thetabar}{{\overline{\theta}}}

\newcommand{\oHbb}{\overline{\Hbb}}
\newcommand{\oC}{\overline{C}}
\newcommand{\zs}{\mathsf{z}}
\newcommand{\ws}{\mathsf{w}}
\newcommand{\onesp}{\overline{\ones}}
\newcommand{\Ri}{\mathsf{R}}
\newcommand{\Si}{\mathsf{S}}
\newcommand{\Ti}{\mathsf{T}}
\newcommand{\epsilonp}{\overline{\epsilon}}

\newcommand{\deltap}{\overline{\delta}}

\newcommand{\ps}{\mathsf{p}}
\newcommand{\qs}{\mathsf{q}}
\newcommand{\Link}{\mathcal{L}}
\newcommand{\Knot}{\mathcal{K}}
\newcommand{\HFKhat}{\widehat{\mathrm{HFK}}}

\newcommand{\var}{\mathsf{u}}

\newcommand{\y}{\mathbf{y}}
\newcommand{\w}{\mathbf{w}}
\newcommand{\z}{\mathbf{z}}
\newcommand\alphas{\mbox{\boldmath$\alpha$}}

\newcommand\betas{\mbox{\boldmath$\beta$}}

\newcommand{\lbl}{\llbracket}
\newcommand{\rbr}{\rrbracket}

\newcommand{\HFhat}{\widehat{\mathrm{HF}}}

\newcommand{\ones}{\mathbb{1}}

\newcommand{\fmapp}{\Fcal}

\newcommand{\fmap}{{\mathfrak{f}}}

\newcommand{\gmap}{{\mathfrak{g}}}

\newcommand{\hmap}{{\mathfrak{h}}}

\newcommand{\oH}{{\overline{\Hcal}}}

\newcommand{\x}{\mathrm{\bf{x}}}

\newcommand{\Sig}{\Sigma}
\newcommand{\ra}{\rightarrow}

\newcommand{\Hbb}{\mathbb{H}}

\newcommand{\Fbb}{\mathbb{F}}

\newcommand{\Fcal}{\mathcal{F}}

\newcommand{\Hcal}{\mathcal{H}}

\newcommand{\Ucal}{\mathcal{U}}

\newcommand{\hcal}{\mathcal{h}}

\newcommand{\mcal}{\mathcal{m}}

\newcommand{\Ker}{\mathrm{Ker}}

\newcommand{\Image}{\mathrm{Im}}

\newtheorem{thm}{\bfseries\color{black!50!blue} Theorem}[section]
\newtheorem{prop}[thm]{\bfseries\color{black!50!blue} Proposition}

\newtheorem{lem}[thm]{\bfseries\color{black!50!blue} Lemma}

\newtheorem{quest}[thm]{\bfseries\color{black!50!blue} Question}
\newtheorem{defn}[thm]{\bfseries\color{black!50!blue} Definition}

\def\endproof{\relax\ifmmode\expandafter\endproofmath\else
	\unskip\nobreak\hfil\penalty50\hskip.75em\hbox{}\nobreak\hfil\bull
	{\parfillskip=0pt \finalhyphendemerits=0 \bigbreak}\fi}
\def\endproofmath$${\eqno\bull$$\bigbreak}
\def\bull{\vbox{\hrule\hbox{\vrule\kern3pt\vbox{\kern6pt}\kern3pt
			\vrule}\hrule}}

\begin{document}
\title{Skein exact triangles in knot Floer homology }%
\author{Eaman Eftekhary}%
\institution{\sf{School of Mathematics, Institute for Research in Fundamental Sciences (IPM), Tehran, Iran}}
\maketitle
\begin{abstract}
	We construct a new family of skein exact triangles for link Floer homology. The skein triples are described by a triple of rational tangles $(R_0,R_1,R_{2n+1})$, where $R_0$ is the trivial tangle and $R_k$ is obtained from it by applying $k$ positive half-twists. We also set up an appropriate framework for potential construction of further skein exact triangles corresponding to arbitrary triples of rational tangles. 
\end{abstract}
\section{Introduction}
Link Floer homology, introduced by Ozsváth–Szabó \cite{OS-knot, OS-link} and Rasmussen \cite{Ras}, categorifies the Alexander polynomial for links in \( S^3 \) (see also \cite{Ef-LFH}).  
To a pointed link $\Link=(L,\ws,\zs)$, consisting of a link $L\subset S^3$ and alternating sets $\ws$ and $\zs$ of marked points on $L$, it assigns a bigraded module $\HFKhat(\Link)$ 
which is equipped with an action of the exterior algebra generated by the markings in $\zs$. 
The kernel of this action, \(\HFKhat(L)\), is a functorially-defined invariant of the underlying link \(L\) \cite{Ef-TQFT}, whose Euler characteristic recovers the symmetrized Alexander polynomial \(\Delta_L(x)\) \cite{OS-link}. This theory has profound applications, including detecting Seifert genus and Thurston norm \cite{OS-Thurston-norm, Ni-Thurston-norm}, 
bounding the $4$-ball genus \cite{OS-four-ball, OSS-upsilon,Hom-Wu:tau} and the unknotting number \cite{AE-unknotting, Ef-TR}, distinguishing fibered knots \cite{Ni-fibered-knots}, and extracting concordance properties \cite{J-concordance, AE-2,Zemke-1}, as well as  studying sutured manifolds \cite{J-sutured, Juh-taut} and surface embeddings \cite{OS-surface, JZ-slice-disks}.\\  

A central feature of link Floer homology is its behavior under local modifications, encoded in skein exact triangles. The oriented skein triangle \cite{OS-Exact-Seq} categorifies the Alexander skein relation \cite{Alexander-skein,Conway-polynomial}. Fix a triple \((K_+, K_-, K_0)\) of knots with identical diagrams outside a disk $D$, where $K_+\cap D$ is a positive crossing, $K_-\cap D$ is a negative crossing and $K_0$ is the oriented resolution of the crossing. If  \((\Knot_+,\Knot_-,\Knot_0)\) is obtained from \((K_+, K_-, K_0)\) by adding the fixed sets $\ws$ and $\zs$ of markings, there is an exact sequence  
\begin{center}  
	\begin{tikzcd}  
		\HFKhat(\Knot_0) \arrow[rr,"\fmap_-"] && \HFKhat(\Knot_+) \arrow[dl,"\fmap_0"] \\  
		& \HFKhat(\Knot_-) \arrow[ul,"\fmap_+"] &  
	\end{tikzcd}  
\end{center}
where \(\fmap_-\) has degree \((-1,0)\) and \(\fmap_0, \fmap_+\) have degree \((0,0)\). This triangle enables inductive calculations \cite{Manolescu-Ozsvath} and underpins surgery formulas \cite{OS-surgery,Ef-surgery,HL-surgery}. It has counterparts in sutured Floer homology, as discussed in \cite[Section 8]{AE-1}. Manolescu's unoriented skein triangle \cite{Ciprian-unoriented} relates the two resolutions  of crossing without consistent orientations. The maps in Manolescu's  triangle preserve Alexander grading, but not the homological grading. This triangle bridges link Floer homology with Khovanov homology \cite{Khovanov} and has been also used to study minimal genus problem for non-orientable surfaces bounding knots \cite{Levine-Zemke}. The maps in the triangle are used in the study of unoriented cobordisms between knots, as considered in \cite{Zemke-nonorientable}. \\

Despite their utility, the known skein triangles are limited to the aforementioned  local modifications. We investigate whether other triples of rational tangles \((R_0, R_1, R_2)\) induce exact triangles in \(\HFKhat\). A \emph{skein triple of type \((R_0, R_1, R_2)\)} consists of three pointed oriented links 
\begin{align*}
(\Knot_0=(K_0,\ws,\zs),\Knot_1=(K_1,\ws,\zs),\Knot_2=(K_2,\ws,\zs)),
\end{align*}
 which are  identical outside a ball \(B\), while $\Knot_\kappa\cap B$ is the rational tangle  \(R_\kappa\) in \(B\) (oriented from points \(\ps\subset \partial B\) to \(\qs\subset\partial B\)). The  markings on $\Knot_\kappa$ are given in the form of two fixed sets \(\ws \supset \ps\) and \(\zs \supset \qs\) (see Figure~\ref{fig:skein}).  
\begin{quest}\label{quest:main}  
	For a fixed triple \((R_0, R_1, R_2)\) of rational tangles, do there exist bigraded \(\Fbb\)-modules \(V_0, V_1, V_2\) such that every skein triple \((\Knot_0, \Knot_1, \Knot_2)\) of this type corresponds to an exact triangle
	\begin{center}  
		\begin{tikzcd}[row sep=large, column sep=large]  
			\HFKhat(\Knot_0) \otimes V_0 \arrow[rr,"\fmap_0"] && \HFKhat(\Knot_1) \otimes V_1 \arrow[dl,"\fmap_1"] \\  
			& \HFKhat(\Knot_2) \otimes V_2 \arrow[ul,"\fmap_2"] &  
		\end{tikzcd}  
	\end{center}
	with \(\fmap_0,\fmap_1\) of degree \((0, 0)\) and \(\fmap_2\) of degree \((-1,0)\)?
\end{quest}

\begin{figure}
	\def\svgwidth{0.8\textwidth}
	{\small{
					\begin{center}
							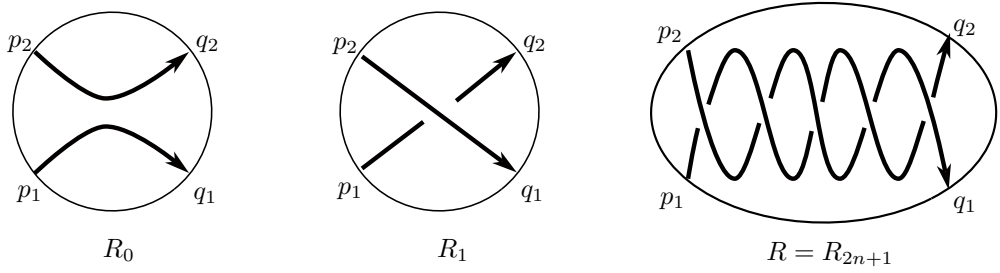
				\end{center}}}
	\caption{The trivial oriented tangle $R_0$ (left), the rational tangle $R_1$ obtained by a positive half-twist from it (middle) and  the rational tangle $R=R_{2n+1}$  obtained from $R_0$ by applying $2n+1$ positive half-twists form an exact triple of rational tangles.}
	\label{fig:skein}
\end{figure} 

Our main results address this question when \(R_0\) is the trivial tangle and \(R_1\) is a positive half-twist. By gluing \((B,R_\ell)\) to \((B,R_\kappa)\) via an orientation-reversing boundary diffeomorphism, we obtain pointed links \(\Link_{\kappa,\ell}\subset S^3\) decorated with the marked points in \(\ps\) and \(\qs\). When \(R_2 = R\) is an exact rational tangle (Definition~\ref{defn:technical-condition}), we show:

\begin{thm}\label{thm:exactness-intro}  
For a triple $(R_0,R_1,R)$ of rational tangles, where $R_0$ is the trivial tangle, $R_1$ is the positive half-twist and $R$ is an exact rational tangle,	
there exist bi-graded $\Fbb$-modules $V_0$ and $V_1$ such that for every skein triple \((\Knot_0, \Knot_1, \Knot_2)\) of type \((R_0, R_1, R)\), there is an exact triangle:
	\begin{center}  
		\begin{tikzcd}[row sep=large, column sep=large]  
			\HFKhat(\Knot_0) \otimes_\Fbb V_0 \arrow[rr,"\fmap"] && \HFKhat(\Knot_1) \otimes_\Fbb V_1 \arrow[dl,"\gmap"] \\  
			& \HFKhat(\Knot_2) \arrow[ul,"\hmap"] &  
		\end{tikzcd}  
	\end{center}  
where the maps $\fmap,\gmap$ and $\hmap$ are homogeneous with respect to the bigrading of degrees $(0,0)$, $(0,0)$ and $(-1,0)$, respectively. 
\end{thm}

A special case of Theorem~\ref{thm:exactness-intro} occurs when \(R\) is obtained by \(2n+1\) half-twists from \(R_0\):

\begin{thm}\label{thm:twist-case-intro}  
	Let \(R_{n}\) be the rational tangle obtained from \(R_0\) by \(n\) positive half-twists and set \(V_k = \bigoplus_{i=1}^k \Fbb \langle v_i \rangle\) with \(\deg(v_i) = (i,i)\). Then for any skein triple \((\Knot_0, \Knot_1, \Knot_2)\) of type \((R_0, R_1, R_{2n+1})\), there is an exact triangle:
	\begin{center}  
		\begin{tikzcd}[row sep=large, column sep=large]  
			\HFKhat(\Knot_0) \otimes V_{2n} \arrow[rr,"\fmap"] && \HFKhat(\Knot_1) \otimes V_{2n+1} \arrow[dl,"\gmap"] \\  
			& \HFKhat(\Knot_2) \arrow[ul,"\hmap"] &  
		\end{tikzcd}  
	\end{center}
where the maps $\fmap,\gmap$ and $\hmap$ are homogeneous with respect to the bigrading.  
\end{thm}

While Theorem~\ref{thm:twist-case-intro} could be proved directly, we derive it from Theorem~\ref{thm:exactness-intro} via Proposition~\ref{prop:exact-triangle-quadruples}, establishing a framework for future generalizations.

\subsection*{Acknowledgments}  
The author thanks Jiajun Wang for discussions and acknowledges Peking University and BICMR for their hospitality.

\section{The Algebraic Foundation}
We begin by establishing the algebraic framework for our skein exact triangles. For any chain map $f:(C,d_C)\to (C',d_{C'})$ between chain complexes, let $[d,f] = f\circ d_C + d_{C'}\circ f$ denote the commutator.

\begin{prop}\label{prop:exact-triangle-criteria}
	Let $(C_\kappa,d_\kappa)$ be chain complexes for $\kappa \in \Z/3 = \{0,1,2\}$ with chain maps:
	\begin{align*}
		f_\kappa:(C_\kappa,d_\kappa)\to (C_{\kappa+1},d_{\kappa+1})
	\end{align*}
	Suppose there exist homomorphisms $H_\kappa:C_{\kappa+1}\to C_\kappa$ and $P_\kappa:C_\kappa\to C_\kappa$ satisfying:
	\begin{align}\label{eq:exact-triangle-criteria-1}
		(1)\ [d,H_\kappa] &= f_{\kappa-1}\circ f_{\kappa+1} \\
		(2)\ [d,P_\kappa] &= \text{Id}_{C_\kappa} + H_\kappa\circ f_\kappa + f_{\kappa-1}\circ H_{\kappa-1}
	\end{align}
	Then the induced maps $\fmap_\kappa:H_*(C_\kappa,d_\kappa)\to H_*(C_{\kappa+1},d_{\kappa+1})$ form an exact triangle:
	\begin{center}
		\begin{tikzcd}
			H_*(C_0,d_0) \arrow[rr,"\fmap_0"] && H_*(C_1,d_1) \arrow[dl,"\fmap_1"] \\
			& H_*(C_2,d_2) \arrow[ul,"\fmap_2"] &
		\end{tikzcd}
	\end{center}
\end{prop}

\begin{proof}
	Assumption (1) implies $\fmap_{\kappa-1}\circ\fmap_\kappa = 0$ for all $\kappa\in\Z/3$. For exactness, consider $z\in C_\kappa$ closed with $\fmap_\kappa[z] = 0$, so $f_\kappa(z) = d_{\kappa+1}(x)$. Using assumption (2):
	\begin{align*}
		z &= d_\kappa(P_\kappa(z)) + H_\kappa(f_\kappa(z)) + f_{\kappa-1}(H_{\kappa-1}(z)) \\
		&= d_\kappa(P_\kappa(z) + H_\kappa(x)) + f_{\kappa-1}(H_{\kappa-1}(z) + f_{\kappa+1}(x))
	\end{align*}
	The element $y = H_{\kappa-1}(z) + f_{\kappa+1}(x)$ is closed since:
	\begin{align*}
		d_{\kappa-1}(y) &= H_{\kappa-1}(d_\kappa(z)) + f_{\kappa+1}(d_{\kappa+1}(x)) = 0
	\end{align*}
	Thus $[y]$ satisfies $\fmap_{\kappa-1}[y] = [z]$, proving exactness.
\end{proof}

\subsection{Heegaard Quadruples and Triangle Maps}
Consider a Heegaard quadruple $\oH = (\Sigma,\alphas,\betas_0,\betas_1,\betas_2,\ws,\zs)$ where:
\begin{itemize}\setlength\itemsep{-0.3em}
	\item The surface $\Sigma$ has genus $g$ with $|\alphas| = |\betas_\kappa| = m + g$;
	\item $|\ws| = |\zs| = m + 1$;
	\item Each component of $\Sigma-\alphas$ and $\Sigma-\betas_\kappa$ has genus zero and contains exactly one marking from  $\ws$ and one marking from $\zs$.
\end{itemize}

 We treat the set $\ws$ of markings as punctures and define the homological grading on the generators using them. The second set $\zs$ of markings gives the Alexander filtration on the generators and recording the coefficients of the holomorphic disks at $\zs$ using the powers of the variable $\var\in \Fbb[\var]$ gives the minus version of the link Floer homology groups. For $\kappa,\ell,\mcal\in\Z/3=\{0,1,2\}$, we obtain the  sub-diagrams
\begin{align*}
&\oH_{\kappa}=(\Sig,\alphas,\betas_{\kappa},\ws,\zs),
 &&\oH_{\kappa,\ell}=(\Sig,\alphas,\betas_{\kappa},\betas_{\ell},\ws,\zs),\\
&\Hcal_{\kappa,\ell}=(\Sig,\betas_{\kappa},\betas_{\ell},\ws,\zs),
&&\Hcal_{\kappa,\ell,\mcal}=	(\Sig,\betas_{\kappa},\betas_{\ell},\betas_{\mcal},\ws,\zs)
\end{align*}	
of $\oH$. The  sub-diagram $\oH_{\kappa}$ determines a pointed link $\Knot_{\kappa}$ in a $3$-manifold $Y_{\kappa}$, which is assumed to be  the $3$-sphere $S^3$. The corresponding minus and hat knot Floer chain complexes are denoted by
 \begin{align*}
 &(C^-_{\kappa},d^-_{\kappa})\quad\text{and}\quad (C_{\kappa},d_{\kappa})=(C^-_{\kappa},d^-_{\kappa})\otimes_{\Fbb[\var]}\Fbb,
 \end{align*} 
respectively. Further, the sub-diagram $\Hcal_{\kappa,\ell}$ corresponds to  a pointed link $\Link_{\kappa,\ell}$ in a $3$-manifold $Y_{\kappa,\ell}$ which is assumed to be $\#^{g}S^1\times S^2$, and determines the minus and hat  Heegaard-Floer chain complexes 
\begin{align*}
&(C^-_{\kappa,\ell},d^-_{\kappa,\ell})\quad\text{and}\quad (C_{\kappa,\ell},d_{\kappa,\ell})=(C^-_{\kappa,\ell},d^-_{\kappa,\ell})\otimes_{\Fbb[\var]}\Fbb.
\end{align*}
 
 Let us denote the homology groups associated with the chain complexes  $(C_{\kappa},d_{\kappa})$, $(C_{\kappa,\ell},d_{\kappa,\ell})$, $(C^-_{\kappa},d^-_{\kappa})$ and $(C^-_{\kappa,\ell},d^-_{\kappa,\ell})$ by $\Hbb_{\kappa}$, $\Hbb_{\kappa,\ell}$, $\Hbb^-_{\kappa}$ and $\Hbb^-_{\kappa,\ell}$, respectively. 
 When $\kappa=\ell$, the homology group $\Hbb_{\kappa,\kappa}$ admits a top generator (with respect to the homological grading) which is denoted by $\ones_{\kappa}$. 
 Note that $\Hbb_{\kappa,\kappa}$ is generated by $\ones_{\kappa}$ under the homology action of 
 \begin{align*}
\wedge_{\kappa}:= \wedge^*H_1(Y_{\kappa,\kappa},L_{\kappa,\kappa})=\wedge^*H_1(\#^g S^1\times S^2, U_{m+1}),
\end{align*} 
where $U_{m+1}$ denotes the $(m+1)$-pointed unlink. 
We may regard the one-dimensional space generated by $\ones_{\kappa}$ as a quotient of $\Hbb_{\kappa,\kappa}$ by a sub-module $\Hbb^\star_{\kappa,\kappa}$ formed by generators in lower homological gradings.\\
 
The differentials $d^-_{\kappa}$ and  $d^-_{\kappa,\ell}$ induce differentials on $\Hbb_{\kappa}$ and $\Hbb_{\kappa,\ell}$, respectively, which are denoted by $\partial_{\kappa}:\Hbb_{\kappa}\ra\Hbb_{\kappa}$ and $\partial_{\kappa,\ell}:\Hbb_{\kappa,\ell}\ra\Hbb_{\kappa,\ell}$, respectively. Note that for $\kappa=\ell$, the latter differential is trivial.
The Heegaard triples $\oH_{\kappa,\ell}$ and $\Hcal_{\kappa,\ell,\mcal}$ determine holomorphic triangle chain maps:
\begin{align*}
	&F_{\kappa,\ell}: C_\kappa \otimes C_{\kappa,\ell} \to C_\ell \quad\text{and}\quad
	F_{\kappa,\ell,m}: C_{\kappa,\ell} \otimes C_{\ell,m} \to C_{\kappa,m}
\end{align*}
with minus versions $F^-_{\kappa,\ell}$ and $F^-_{\kappa,\ell,m}$. We simplify notation by writing $d$, $\partial$, $F$, etc., when indices are clear from context.
By further abuse of notation, the maps induced  by $F$ and $F^-$ in homology are also denoted by $F$ and $F^-$.   Note that $F(\cdot\otimes \ones_{\kappa}):\Hbb_{\kappa}\ra\Hbb_{\kappa}$ is the identity map. Rectangle and pentagon counts give maps $H,H^-,P,P^-$ satisfying the key identities
\begin{align*}
	&[d,H](\x\otimes\y\otimes \w)=F(\x\otimes F(\y\otimes\w))+F(F(\x\otimes\y)\otimes\w)\quad\text{and}\\
	&[d,P](\x\otimes\y\otimes\w\otimes\z)=
	H(\x\otimes\y\otimes F(\w\otimes\z))+H(\x\otimes F(\y\otimes \w)\otimes\z)+H(F(\x\otimes\y)\otimes \w\otimes\z)\\
	&\quad\quad\quad\quad\quad\quad\quad\quad\quad
	\quad\quad\quad\quad\quad\quad\quad\quad
	+F(\x\otimes H(\y\otimes\w\otimes\z))+F(H(\x\otimes\y\otimes\w)\otimes\z)).
	\numberthis\label{eq:compositios-law}
\end{align*}	 
for $\x,\y,\w$ and $\z$ in appropriate chain complexes. Similar identities are satisfied with $d^-,F^-,H^-$ and $P^-$ in place of $d,F,H$ and $P$, respectively.  

\subsection{Exact Triples}
Fix the disjoint index sets $I_{\kappa}$ for $\kappa\in \Z/3$  (the disjointness assumption is for notational simplicity and is clearly not important), and the sets $T_{\kappa}=\{\theta^j_k\}_{j\in I_{\kappa},k\in I_{\kappa+1}}$ of closed generators from $C_{\kappa,\kappa+1}$.  Assume that $(C_{\kappa}^j,d_{\kappa}^j)$ is a copy of $(C_{\kappa},d_{\kappa})$, with an appropriate shift $\lbl m_{\kappa}^j,a_{\kappa}^j\rbr$ in the homological/Alexander bigrading, so that
\begin{align*}
	m_{\kappa}^j=m_{\kappa+1}^k+M(\theta^j_k)-\delta_{\kappa}^2\quad\text{and}\quad a_{\kappa}^j=a_{\kappa+1}^k+A(\theta^j_k)	\quad\forall\ \kappa\in\Z/3,\ j\in I_{\kappa},\ k\in I_{\kappa+1}.
\end{align*}	
 Therefore, the distinguished generator of $(C_{\kappa}^j,d_{\kappa}^j)$ is in bigrading $(m_{\kappa}^j,a_{\kappa}^j+\tau(\Knot_{\kappa}))$. We then set 
\begin{align*}
(\oC_{\kappa},d_{\kappa})=\bigoplus_{j\in I_{\kappa}}(C_{\kappa}^j,d_{\kappa}^j),\quad\forall\ \kappa\in\Z/3=\{0,1,2\}.
\end{align*}
 We abuse the notation by denoting the differentials of $\oC_{\kappa}$ and $C_{\kappa}$ by the same notation $d_{\kappa}$. 
Define the chain maps $f_{\kappa}:C_{\kappa}\ra C_{\kappa+1}$ by setting the component of $f_{\kappa}:C_{\kappa}\ra C_{\kappa+1}$ from $C_{\kappa}^j$ to $C_{\kappa+1}^k$ 
equal to $F(\cdot\otimes \theta^j_k)$. The chain maps $f^j_k$ and $f_{\kappa}$  induce the maps $\fmap^j_k$ and $\fmap_{\kappa}$ in homology, respectively.  Moreover, $\fmap_0$ and $\fmap_1$ preserve the bigrading while $\fmap_2$ drops the homological grading by $1$ and preserves the Alexander grading. For $\kappa\in\Z/3$  set
\begin{align}\label{eq:Theta-I-defn}
	\Theta^j_l:=\sum_{k\in I_{\kappa}}\theta^j_k\otimes \theta^k_l,\quad \forall\ j\in I_{\kappa}, l\in I_{\kappa-1}
	\quad\quad\text{and}\quad\quad
	\Lambda^j_{i}:=\sum_{k\in I_{\kappa+1}, l\in I_{\kappa-1}} \theta^j_k\otimes\theta^k_l\otimes\theta^l_{i}
	\quad\forall i,j\in I_{\kappa}.
\end{align}
\begin{defn}
	The triple $(T_0,T_1,T_2)$ has \emph{trivial compositions} if for every $j\in I_\kappa$ and $l\in I_{\kappa+1}$ we have $F(\Theta^j_l) = d\eta^j_l$ for some $\eta^j_l$. It is \emph{exact} if additionally, for every $i,j\in I_\kappa$  we have 
	\begin{align}\label{eq:Delta-defn}
		\Delta^j_i := H(\Lambda^j_i) + \sum_k F(\theta^j_k \otimes \eta^k_i) + \sum_l F(\eta^j_l \otimes \theta^l_i) = \delta^j_i \cdot \ones_\kappa \mod \Hbb^\star_{\kappa,\kappa}.
	\end{align}
\end{defn}

With this setup and the above definitions in place, we are now ready to formulate a version of Proposition~\ref{prop:exact-triangle-criteria} for Heegaard qudruples.
\begin{prop}\label{prop:exact-triangle-quadruples}
 For an exact triple $(T_0,T_1,T_2)$ of generator sets,
  there exists an exact triangle:
\begin{center}
\begin{tikzcd}[ row sep=large, column sep=large]
\oHbb_0=\bigoplus_{j\in I_0}\Hbb_0^j\arrow[rr,"\fmap_0=(\fmap^j_k)_{j\in I_0,k\in I_1}"]
&&\oHbb_1=\bigoplus_{j\in I_1}\Hbb_1^j\arrow[dl,"\fmap_1=(\fmap^j_k)_{j\in I_1,k\in I_2}"]\\
&\oHbb_2=\bigoplus_{j\in I_2}\Hbb_2^j\arrow[ul,"\fmap_2=(\fmap^j_k)_{j\in I_2,k\in I_0}"]&
\end{tikzcd}.
\end{center}	
\end{prop}		
\begin{proof}
It suffices to show that the two conditions in Proposition~\ref{prop:exact-triangle-criteria} are satisfied for appropriate maps $H_{\kappa}:\oC_{\kappa+1}\ra \oC_{\kappa}$ and $P_{\kappa}:\oC_{\kappa}\ra \oC_{\kappa}$. For this purpose, let us assume that $F(\Theta^j_l)=d\eta^j_l$ for every $\kappa\in\Z/3$, $j\in I_{\kappa+1}$ and $l\in I_{\kappa}$, define $H_{\kappa}$ by setting its restriction from $C_{\kappa+1}^j$ to $C_{\kappa}^l$ equal to 
\begin{align*}
H^j_l=H(\cdot\otimes\Theta^j_l)+F(\cdot\otimes \eta^j_l),\quad\forall\ (j,l)\in I_{\kappa+1}\times I_{\kappa}.
\end{align*}
Since $H(\Lambda^j_{i})=\delta^j_{i}\ones_{\kappa}$ modulo $\Hbb_{\kappa,\kappa}^\star$, and since the holomorphic triangle map $F:C_{\kappa}\otimes\Hbb^\star_{\kappa,\kappa}\ra C_{\kappa}$ is chain homotopic to zero, there are maps $Q^j_{i}:C_{\kappa}\ra C_{\kappa}$ so that 
\begin{align*}
\delta^j_{i} Id_{C_{\kappa}}&=[d,Q^j_{i}]+F(\cdot\otimes \Delta^j_{i})\\
&=[d,Q^j_{i}]+F(\cdot\otimes H(\Lambda^j_{i}))+\sum_{k\in I_{\kappa+1}}F(\cdot\otimes F(\theta^j_k\otimes\eta^k_{i}))+\sum_{l\in I_{\kappa-1}}F(\cdot\otimes F(\eta^j_l\otimes\theta^l_{i})).	
\end{align*}	
We may then define $P_{\kappa}:\oC_{\kappa}\ra\oC_{\kappa}$ by setting its restriction from $C_{\kappa}^j$ to $C_{\kappa}^{i}$ 	
equal to 
\begin{align*}
P^j_{i}=P(\cdot\otimes\Lambda^j_{i})+\sum_{l\in I_{\kappa-1}}H(\cdot\otimes \eta^j_l\otimes\theta^l_{i})
+\sum_{k\in I_{\kappa+1}} H(\cdot\otimes \theta^j_k\otimes\eta^k_{i})
+Q^j_{i}.
\end{align*}
Here, $P$ denotes the pentagon map associated with the Heegaard $5$-tuple
$(\Sig,\alphas,\betas_{\kappa},\betas_{\kappa+1},\betas_{\kappa-1},\betas_{\kappa},\ws,\zs)$. It then follows that for $\x\in C_{\kappa}^j$, the $l$-th component $[d,H_{\kappa}]^l(\x)$ of $[d,H_{\kappa}](\x)$ is given by
\begin{align*}
[d,H_{\kappa}]^l(\x)&=[d,H](\x\otimes \Theta^j_l)+[d,F](\x\otimes \eta^j_l)+F(\x\otimes d(\eta^j_l))&&\\
&=F(\x\otimes (F(\Theta^j_l)+d(\eta^j_l)))+\sum_{k\in I_{\kappa+1}}F(F(\x\otimes \theta^j_k)\otimes\theta^k_l)&&\text{since }[d,F]=0,\\
&=	\sum_{k\in I_{\kappa+1}}f^k_l(f^j_k(\x))&&\text{since } F(\Theta^j_l)=d(\eta^j_l)\\
&=(f_{\kappa+1}\circ f_{\kappa})^l(\x),
\end{align*}	 
where $(\fmap_{\kappa+1}\circ\fmap_{\kappa})^l(\x)$ denotes the $l$-th component of $(\fmap_{\kappa+1}\circ\fmap_{\kappa})(\x)$. This proves that condition (1) from (\ref{eq:exact-triangle-criteria-1}) is satisfied. Similarly, note that
\begin{align*}
	[d,P_{\kappa}]^{i}(\x)&=[d,P](\x\otimes \Lambda^j_l)+\sum_{l\in I_{\kappa-1}}[d,H](\x\otimes\eta^j_l\otimes\theta^l_{i})
	+\sum_{k\in I_{\kappa+1}}[d,H](\x\otimes \theta^j_k\otimes\eta^k_{i})&&\\
	&\quad\ \ +[d,Q^j_{i}](\x)+\sum_{l\in I_{\kappa-1}}H(\x\otimes d(\eta^j_l)\otimes\theta^l_{i})
	+\sum_{k\in I_{\kappa+1}} H(\x\otimes \theta^j_k\otimes d(\eta^k_{i}))&&\\
	&=\delta^j_{i}Id_{C_{\kappa}}(\x)+\sum_{k\in I_{\kappa+1}}H(F(\x\otimes \theta^j_k)\otimes\Theta^k_{i})+\sum_{l\in I_{\kappa-1}}F(H(\x\otimes \Theta^j_l)\otimes \theta^l_{i})&&\\
	&\quad\quad+\sum_{l\in I_{\kappa-1}}H(\x\otimes (F(\Theta^j_l)+d(\eta^j_l))\otimes\theta^l_{i})
	+\sum_{k\in I_{\kappa+1}} H(\x\otimes \theta^j_k\otimes (F(\Theta^k_{i})+d(\eta^k_{i})))&&\\ 
	&\quad\quad+\sum_{l\in I_{\kappa-1}}F(F(\x\otimes \eta^j_l)\otimes\theta^l_{i})
	+\sum_{k\in I_{\kappa+1}} F(F(\x\otimes \theta^j_k)\otimes \eta^k_{i})&&\\
	&=(\delta^j_{i}Id_{C_{\kappa}}+H_{\kappa}\circ f_{\kappa}+f_{\kappa-1}\circ H_{\kappa-1})^{i}(\x).
\end{align*}
Therefore,  criteria (2) from (\ref{eq:exact-triangle-criteria-1}) is also satisfied.
This completes the proof of the proposition.
\end{proof}	 
\section{Skein Triples of Knot Diagrams and Link Floer Homology}\label{sec:skeins}
Let  $K$ be a knot in $S^3$ and $B\subset S^3$ a ball so that  $\partial B$  intersects $K$ transversely in $4$  points which are grouped as $\ps=\{p_1,p_2\}$ and  $\qs=\{q_1,q_2\}$. We  assume  $K\cap B$ consists of two oriented disjoint arcs  from $\ps$ to $\qs$, forming a $2$-tangle $R$ in  $B$. 
A \emph{replacement} changes $K$ to $K'$ by substituting $(B,R)$ with another $2$-tangle $(B,R')$ with $\partial R'= \ps\cup\qs$. We consider orientation-preserving replacements where strands in $R'$ go from $\ps$ to $\qs$. When both $R$ and $R'$ are rational tangles (obtained from the trivial tangle $(B,R_0)$ by applying diffeomorphisms of $B$ fixing $\ps\cup\qs$), we call this a \emph{rational replacement}. 
For such a replacement, the union
\begin{align*}
	(S^3,L_{R',R}) = (B,R')\cup\overline{(B,R)}
\end{align*}
forms an oriented link called the \emph{characteristic link}, with marked points $\ps$ and $\qs$ giving a 4-pointed link $\Link=(S^3,L,\ps,\qs)$.\\

\subsection{Exact Triples  of Rational Tangles}
A \emph{skein triple} $(K_0,K_1,K_2)$ consists of three oriented links identical outside $B\subset S^3$, where each $K_\kappa\cap B$ is a rational 2-tangle $R_\kappa$ connecting $\ps$ to $\qs$ ($\kappa\in\Z/3$). 
Extending $\ps,\qs$ to decorations $\ws=\ps\cup\ws_0$ and $\zs=\qs\cup\zs_0$ (with $|\ws_0|=|\zs_0|=m-1$) gives pointed links $\Knot_\kappa=(K_\kappa,\ws,\zs)$. The characteristic links are:
\begin{align*}
	L_{\kappa,\ell} = (B,R_\kappa)\cup\overline{(B,R_{\ell})}
\end{align*}
with pointed versions $\Link_{\kappa,\ell}$, which are decorated by $\ps$ and $\qs$. We assume that $L_{\kappa,\ell}$ is not the $2$-component unlink unless $\kappa=\ell$.\\

Fix a Heegaard quadruple $\oH=(\Sigma,\alphas,\betas_0,\betas_1,\betas_2,\ws,\zs)$ for $(\Knot_0,\Knot_1,\Knot_2)$, where $\Sigma$ has genus zero, $\oH_\kappa$ represents $\Knot_\kappa$ and $\Hcal_{\kappa,\ell}$ represents $\Link_{\kappa,\ell}\sqcup \Ucal_{m-1}$, where $\Ucal_{k}$ denotes a $k$-component unlink with a pair of markings on each one of its components. 
Since $R_0,R_1,R_2$ are rational, we can assume $\betas_\kappa=\hcal_\kappa(\betas)\cup\{\beta_\kappa\}$ with $\beta_\kappa$ splitting $\Sigma$ into two disks and separating the points in either of $\ps$ and  $\qs$ and $\hcal_\kappa$ a small an generic Hamiltonian isotopy. We abuse the notation by dropping $\hcal_\kappa$ and writing  $\betas_\kappa=\betas\cup\{\beta_\kappa\}$.\\

Define chain complexes $C_\kappa, C^-_\kappa, C_{\kappa,\ell}, C^-_{\kappa,\ell}$ as before, with hat differential $d=d_\bullet$ on $C_\bullet$, minus differential $d^-=d^-_\bullet$ on $C^-_\bullet$, triangle maps $F,F^-$ and rectangle maps $H,H^-$.
For appropriate $\betas_\kappa$, we have $d_{\kappa,\ell}=0$ and $d^-_{\kappa,\ell}=\var\cdot\partial_{\kappa,\ell}$. The modules satisfy:
\begin{align*}
	\Hbb_{\kappa,\ell} &= C_{\kappa,\ell} = \HFKhat(\Link_{\kappa,\ell})\otimes(\Fbb\oplus\Fbb\lbl 1,1\rbr)^{m-1} \quad\text{and}\quad
	H_*(C_{\kappa,\ell},\partial)= \HFhat(S^3,\ws) = (\Fbb\oplus \Fbb\lbl 1\rbr)^m
\end{align*}
The generators $\ones_{\kappa,\ell}, \onesp_{\kappa,\ell}$ correspond to the top and bottom generators of 
\begin{align*}
H_*(\HFKhat(\Link_{\kappa,\ell}),\partial)=\Fbb\oplus\Fbb\lbl 1\rbr,
\end{align*}
 paired with the top generator of $(\Fbb\oplus\Fbb\lbl 1,1\rbr)^{m-1}$.
If $\kappa,\ell\in\Z/3$ are different, the image of $\ones_{\kappa,\ell}\otimes\ones_{\ell,\kappa}$ under $F^-:\Hbb^-_{\kappa,\ell}\otimes_{\Fbb[\var]} \Hbb^-_{\ell,\kappa}\ra \Hbb^-_{\kappa,\kappa}$ is of the form $\var^{r(\kappa,\ell)}\ones_{\kappa}$. If $r(\kappa,\ell)=0$, pairing with $\ones_{\kappa,\ell}$ and $\ones_{\ell,\kappa}$ gives maps 
\begin{align*}
	\fmap:\HFKhat(\Link_{\ell,\kappa})\ra\HFKhat(\Link_{\ell,\ell})=\left(\Fbb\oplus\Fbb\lbl 1,1\rbr\right)^m
	\quad\text{and}\quad
	\gmap:\HFKhat(\Link_{\ell,\ell})=\left(\Fbb\oplus\Fbb\lbl 1,1\rbr\right)^m	\ra \HFKhat(\Link_{\ell,\kappa})
\end{align*}	
with $\gmap\circ\fmap=Id$. In particular, the homomorphism $\fmap$ is injective, which can not happen for $\kappa\neq \ell$. Therefore, $F(\ones_{\kappa,\ell}\otimes\ones_{\ell,\kappa})=0$, or equivalently $r(\kappa,\ell)>0$. Since $R$ is a rational tangle, it follows from \cite[Lemma 2.1]{Ef-RTR} that $r(a,b)\in\{0,1\}$ and  $F^-(\ones_{\kappa,\ell}\otimes\ones_{\ell,\kappa})=\var\cdot\ones_\kappa$.\\

	It follows from (\ref{eq:compositios-law}) for the minus-theory that $\partial\circ F=F\circ\partial$ is satisfied for all relevant triangle maps $F$ in the level of homology and if $\kappa,\ell,\mcal\in\Z/3$  and $H=H_{\kappa,\ell,\mcal,\kappa}$ is a corresponding holomorphic rectangle map, there are homomorphism $F':C_{\kappa,\ell}\otimes_\Fbb C_{\ell,\mcal}\ra  C_{\kappa,\mcal}$ (defined by counting holomorphic triangles with multiplicity $1$ at $\qs$) so that
\begin{align*}
	(\partial H+H\partial)(\x\otimes\y\otimes \w)
	&=F'\big(\x\otimes F(\y\otimes\w)\big)+F'\big(F(\x\otimes\y)\otimes\w\big)\\
	&\quad\quad\quad+F\big(\x\otimes F'(\y\otimes\w)\big)+F\big( F'(\x\otimes\y)\otimes\w\big).
	\numberthis\label{eq:compositios-law-2}
\end{align*}	
\begin{lem}\label{lem:vanifhing-of-F1}
	If $\kappa,\ell,\mcal\in\Z/3$ are different and $H=H_{\kappa,\ell,\mcal,\kappa}$ is the corresponding holomorphic rectangle map, for every 
	$\x\in \Hbb_{\kappa,\ell}=C_{\kappa,\ell},\y\in \Hbb_{\ell,\mcal}=C_{\ell,\mcal}$ and $\z\in \Hbb_{\mcal,\kappa}=C_{\mcal,\kappa}$ 
	we have 
	\begin{align*}
		(\partial H+H\partial)(\x\otimes\y\otimes \w)
		&=F'\big(\x\otimes F(\y\otimes\w)\big)+F'\big(F(\x\otimes\y)\otimes\w\big)
		\quad\quad{\text{modulo}\ \Hbb_{\kappa,\kappa}^\star}.
		\numberthis\label{eq:compositios-law-3}
	\end{align*}		
\end{lem}	
\begin{proof}
	Given (\ref{eq:compositios-law-2}), since $d_{\kappa,\kappa}=0$, it suffices to prove that 
	\begin{align*}
		F\big(\x\otimes F'(\y\otimes\w)\big)+F\big( F'(\x\otimes\y)\otimes\w\big)=0\quad\quad{\text{modulo}\ \Hbb_{\kappa,\kappa}^\star}.	
	\end{align*}	
	Note that $\ones_{\kappa}$ is neither in the image of $F:\Hbb_{\kappa,\ell}\otimes \Hbb_{\ell,\kappa}\ra \Hbb_{\kappa,\kappa}$ nor in the image of $F:\Hbb_{\kappa,\mcal}\otimes \Hbb_{\mcal,\kappa}\ra \Hbb_{\kappa,\kappa}$. 
	This completes the proof of the lemma. 
\end{proof}	

 Let us denote $I_0,I_1$ and $I_2$ by $I,J$ and $K$, respectively and use $i,i'$, etc. to denote elements of $I$, use $j,j'$, etc. to denote elements of $J$ and use $k,k'$, etc. to denote elements of $K$. 
\begin{defn}\label{defn:special-triples}
The triple $(T_0,T_1,T_2)$ of closed generators, with 
\begin{align*}
T_0=\{\theta^i_j=\deltap^i_j\onesp_{0,1}+\epsilonp^i_j\ones_{0,1}\}_{i\in I,j\in J},\quad 
T_1=\{\theta^j_k\}_{j\in J,k\in K}\quad\text{and}\quad T_2=\{\theta^k_i=\partial\thetabar^k_i\}_{k\in K,i\in I}
\end{align*}
 is called a {\emph{special}} triple if for every $j\in J,k\in K$,  $\theta^j_k\in\Hbb_{1,2}$ is $\partial$-closed.
\end{defn}	

We would like to examine the exactness condition, used in Proposition~\ref{prop:exact-triangle-quadruples}, for special triples of closed generators. Set
\begin{align*}
&\Ri^k_j:=\sum_{i\in I}\big(\deltap^i_jF(\thetabar^k_i\otimes \onesp_{0,1})+\epsilonp^i_jF(\thetabar^k_i\otimes\ones_{0,1})\big)\in\Hbb_{2,1}\\	
&\Si^i_k:=\sum_{j\in J}\big(\deltap^i_jF(\onesp_{0,1}\otimes \theta^j_k)+\epsilonp^i_jF(\ones_{0,1}\otimes\theta^j_k)\big)\in\Hbb_{0,2}\\	
&\Ti^j_i:=\sum_{k\in K}F(\theta^j_k\otimes\thetabar^k_i)\in\Hbb_{1,0}.
\end{align*}	
The triple has trivial compositions if $\Si^i_k=0$, while $\Ri^k_j$ and $\Ti^j_i$ are $\partial$-closed for every $i\in I,j\in J$ and $k\in K$. Moreover, since $\partial_{\kappa,\kappa}=0$ for all $\kappa\in\Z/3$, Lemma~\ref{lem:vanifhing-of-F1} implies that
\begin{align*}
\Lambda^i_{i'}&=\sum_{k\in K}F'(\Si^i_k\otimes \thetabar^k_{i'})+\sum_{j\in J}\big(\deltap^i_jF'(\onesp_{0,1}\otimes \Ti^j_{i'})+\epsilonp^i_j F'(\ones_{0,1}\otimes \Ti^j_{i'})\big)\\
&=\sum_{j\in J}\big(\deltap^i_jF'(\onesp_{0,1}\otimes \Ti^j_{i'})+\epsilonp^i_j F'(\ones_{0,1}\otimes \Ti^j_{i'})\big)\\
\Lambda^j_{j'}&=\sum_{k\in K}F'(\theta^j_k\otimes \Ri^k_{j'})+\sum_{i\in I}\big(\deltap^i_{j'}F'(\Ti^j_{i'}\otimes \onesp_{0,1})+\epsilonp^i_{j'} F'( \Ti^j_{i'}\otimes \ones_{0,1})\big)\\
\Lambda^k_{k'}&=\sum_{j\in J}F'(\Ri^k_j\otimes \theta^j_{k'})+\sum_{i\in I}F'(\thetabar^k_i\otimes \Si^i_{k'})\\
&=\sum_{j\in J}F'(\Ri^k_j\otimes \theta^j_{k'})
\end{align*}	
\subsection{Exact Rational Tangles and Triangle Maps}
We assume:
\begin{itemize}\setlength\itemsep{-0.3em}
	\item $R_0$ is the standard trivial tangle connecting $p_1$ to $q_1$ and $p_2$ to $q_2$
	\item $R_1$ is obtained from $R_0$ by a positive half-twist, connecting $p_1$ to $q_2$ and $p_2$ to $q_1$
	\item $R=R_2$ is an arbitrary rational tangle connecting $\ps$ to $\qs$
\end{itemize}

\begin{defn}\label{defn:technical-condition}
	A rational tangle $R$ is \emph{exact} if the map
	\begin{align*}
		F(\cdot\otimes \ones_{0,1}):\HFKhat(\Link_{2,0})\to\HFKhat(\Link_{2,1})
	\end{align*}
	coming from the Heegaard triple $\Hcal_{2,0,1}=(\Sig,\beta_2,\beta_0,\beta_1,\ps,\qs)$ induces an isomorphism
	\begin{align*}
		f:\frac{\HFKhat(\Link_{2,0})}{\Ker(\partial_{2,0})}\to\frac{\HFKhat(\Link_{2,1})}{\langle \Image(\partial_{2,1}), \ones_{1,2}\rangle_\Fbb}\simeq \frac{\Ker(\partial_{1,2})}{\langle\ones_{1,2}\rangle}.
	\end{align*}
\end{defn}

For $R$ exact, choose a basis $\{\theta_i\}_{i\in I}$ for $\HFKhat(\Link_{2,0})/\Ker(\partial)$, giving:
\begin{align*}
	C_{2,0} &= \langle \theta_i, \thetabar_i=\partial \theta_i \mid i\in I\rangle_\Fbb \oplus \langle \ones_{2,0},\onesp_{2,0}\rangle_\Fbb.
\end{align*}
The dual differential $\partial^*=\partial_{0,2}$ satisfies $\partial^*\thetabar_i=\theta_i$. Setting $\theta^i=f(\theta_i)$, we have:
\begin{align*}
	\Ker(\partial_{1,2}) &= \langle \theta^i \mid i\in I\rangle_\Fbb \oplus \langle \ones_{1,2}=:\theta^0\rangle.
\end{align*}
Assume $1\in I$ with $\theta^1=\onesp_{1,2}$. This gives a decomposition:
\begin{align*}
	\HFKhat(\Link_{1,2}) &= \langle \theta^i \mid i\in I\cup\{0\}\rangle_\Fbb \oplus \langle \thetabar^i=\partial^* \theta^i \mid i\in I\setminus\{1\}\rangle_\Fbb.
\end{align*}
The  cobordism associated with $\Hcal_{2,0,1}=(\Sig,\beta_2,\beta_0,\beta_1,\ps,\qs)$  determines a holomorphic triangle map
	\begin{align*}
		\fmapp:\HFKhat(\Link_{0,1})\otimes\HFKhat(\Link_{1,2})\otimes\HFKhat(\Link_{2,0})\ra \Fbb.
\end{align*}
Note that $\fmapp\circ \partial=0$. In the above decompositions for $\HFKhat(\Link_{0,1})$ and $\HFKhat(\Link_{1,2})$, we may choose $\theta_i,\theta^j,\ones_{2,0}$ and $\onesp_{2,0}$ so that the cobordism map $\fmapp$ is given as follows:
\begin{align*}
	&(1)\ \  \fmapp(\ones_{0,1}\otimes \theta^j\otimes \bullet)=\epsilon^i_j\delta^\bullet_{\theta_i}+\delta^1_j\delta^{\bullet}_{\ones_{2,0}},
	&&\text{for}\  i,j\in I,\ \bullet\in\{\ones_{2,0},\onesp_{2,0},\theta_i,\partial \theta_i\},\ \text{and some}\ \epsilon^i_j\in\Fbb,
	\\
	&(2)\ \ \fmapp(\ones_{0,1}\otimes \theta^0\otimes \bullet)=\delta^\bullet_{\onesp_{2,0}},		
	&&\text{for}\  i\in I,\  \bullet\in\{\ones_{2,0},\onesp_{2,0},\theta_i,\partial \theta_i\},	
	\\
	&(3)\ \  \fmapp(\onesp_{0,1}\otimes \theta^j\otimes \bullet)=\delta^\bullet_{\theta_j},
	&&\text{for}\   i,j\in I,\  \bullet\in\{\ones_{2,0},\onesp_{2,0},\theta_i,\partial \theta_i\}\quad\text{and}
	\\
	&(4)\ \ \fmapp(\onesp_{0,1}\otimes \theta^0\otimes \bullet)=\delta^\bullet_{\ones_{2,0}},		
	&&\text{for}\  i\in I,\  \bullet\in\{\ones_{2,0},\onesp_{2,0},\theta_i,\partial \theta_i\}
	\numberthis\label{eq:triangle-count-1}
\end{align*}
Here $\delta^\bullet_\kappa$ denotes the Kronecker delta function which is zero unless $\bullet=\kappa$, when it gives $1$. 
The entries $\{\epsilon^i_j\}_{i,j\in I}$ form a matrix $\Eps$. In the above basis, we may also record other evaluations of $\fmapp$:
\begin{align*}	
	&(1)\ \  \fmapp(\ones_{0,1}\otimes \thetabar^j\otimes \bullet)=\epsilonp^i_j\delta^\bullet_{\theta_i}+\epsilon_j\delta^{\bullet}_{\ones_{2,0}}+
	\epsilonp_j\delta^{\bullet}_{\onesp_{2,0}},
	&&\text{for}\  i,j\in I,\ j\neq 1,\ \bullet\in\{\ones_{2,0},\onesp_{2,0},\theta_i\}\quad\text{and}
	\\
	&(2)\ \  \fmapp(\onesp_{0,1}\otimes \thetabar^j\otimes \bullet)=\deltap^i_j\delta^\bullet_{\theta_i}+\delta_j\delta^{\bullet}_{\ones_{2,0}}+
	\deltap_j\delta^{\bullet}_{\onesp_{2,0}},
	&&\text{for}\  i,j\in I,\ j\neq 1,\ \bullet\in\{\ones_{2,0},\onesp_{2,0},\theta_i\},\numberthis\label{eq:triangle-count-2}
\end{align*}
for some $\epsilon^i_j,\epsilon_j,\epsilonp_j,\deltap^i_j,\delta_j,\deltap_j\in\Fbb$. We set $\epsilon^i_0=0$ for $i\in I$.
It follows from the above computations that
\begin{align}\label{eq:F-v0-theta-infty}
	F(\theta_i\otimes \ones_{0,1})=\sum_{k\in I}\epsilon^i_k \theta^k+\sum_{k\in I\setminus\{1\}}
	\epsilonp^i_k \thetabar^k
	\quad\text{and}\quad
	F(\theta_i\otimes \onesp_{0,1})=\sum_{k\in I}\delta^i_k \theta^k+\sum_{k\in I\setminus\{1\}}
	\deltap^i_k\thetabar^k.
\end{align}	
Moreover, for $j\in I\cup\{0\}$ we find
\begin{align}\label{eq:F-theta-infty-v+}
	&F(\ones_{0,1}\otimes \theta^j)=\delta^0_j\onesp_{2,0}+\delta^1_j\ones_{2,0}+\sum_{k\in I}\epsilon^k_j\theta_k
	\quad\text{and}\quad F(\onesp_{0,1}\otimes \theta^j)=\delta^0_j\ones_{2,0}+\sum_{k\in I}\delta^k_j\theta_k.
\end{align}	
The last computation which follows from the above considerations and is needed in the following discussions is the following:
\begin{align}\label{eq:F-v+-v0}
	&F(\theta^j\otimes \theta_i)=\epsilon^i_j\ones_{0,1}+\delta^i_j\onesp_{0,1},
	\quad\quad\forall\ i\in I,\ j\in I\cup\{0\}.
\end{align}

\section{Exact Triangles for Exact Skeins}\label{sec:exact-triangles}

\subsection{Setup, Notation and Key Lemmas}
We apply Proposition~\ref{prop:exact-triangle-quadruples} to the Heegaard quadruple $\oH = (\Sig,\alphas,\betas_0,\betas_1,\betas_2,\ws,\zs)$ associated with an exact rational tangle. Here we have:

\begin{itemize}\setlength\itemsep{-0.3em}
	\item {\bf Index sets:} $I_2 = \{\star\}$ (single element), $I_0 = I$, and $I_1 = J = I \cup \{0\}$;
	\item  {\bf Generators:}
\begin{align*}
	&(1)\ \theta_i^\star := \thetabar_i = \partial\theta_i,
	&&(2)\ \theta_\star^j := \theta^j\quad\text{and}
	&&(3)\ \theta^i_j := \delta^j_i\ones_{0,1} + (\epsilon^j_i + \delta^1_i\delta^j_0)\onesp_{0,1},  &&\forall i \in I, j \in J.
\end{align*}	
	\item  {\bf Vector spaces:} $V_2$ is the trivial vector space $\Fbb$, while
	\begin{align*}
		V_0 := \langle \theta_i \mid i \in I \rangle_\Fbb \quad\text{and}\quad
		V_1 := \langle \theta^j \mid j \in J \rangle_\Fbb
	\end{align*}
\end{itemize}
This gives us the triangle of maps:
\begin{center}
	\begin{tikzcd}[row sep=large, column sep=large]
		\HFKhat(\Knot_0)\otimes V_0 \arrow[rr,"\fmap"] && \HFKhat(\Knot_1)\otimes V_1 \arrow[dl,"\gmap"] \\
		& \HFKhat(\Knot_2) \arrow[ul,"\hmap"] &
	\end{tikzcd}
\end{center}
where the homomorphisms are defined explicitly by:
\begin{align*}
	\fmap(\x_0 \otimes \theta_i) &= \delta^1_i F(\x_0 \otimes \onesp_{0,1}) \otimes \theta^0 
	+ \sum_{j \in I} \left(\delta^j_i F(\x_0 \otimes \ones_{0,1}) + \epsilon^j_i F(\x_0 \otimes \onesp_{0,1})\right) \otimes \theta^j \\
	\gmap(\x_1 \otimes \theta^j) &= F(\x_1 \otimes \theta^j) \\
	\hmap(\x_2) &= \sum_{i \in I} F(\x_2 \otimes \partial\theta_i) \otimes \theta_i
\end{align*}

\begin{lem}\label{lem:1st-composition-is-trivial}
	For all $i \in I$ and $j \in J$, we have $F(\theta^j \otimes \thetabar_i) = 0$.
\end{lem}

\begin{proof}
	Since $\partial_{1,0} = 0$, we compute:
	\begin{align*}
		F(\theta^j \otimes \thetabar_i) = \partial_{1,0} F(\theta^j \otimes \theta_i) + d_{1,0} F'(\theta^j \otimes \theta_i) = 0.
	\end{align*}
\end{proof}

\begin{lem}\label{lem:2-3-compositions-are-trivial}
	For an exact rational tangle $R$:
	\begin{align*}
		\sum_{i \in I} F(\thetabar_i \otimes \theta^i_j) &= 0 \quad \forall j \in J \quad\quad\text{and}\quad\quad
		\sum_{j \in J} F(\theta^i_j \otimes \theta^j) = 0 \quad \forall i \in I
	\end{align*}
\end{lem}

\begin{proof}
	For the first identity, exactness implies:
	\begin{align*}
		\sum_{i \in I} F(\thetabar_i \otimes \theta^i_j) &= \partial_{2,1}\left(\sum_{i \in I} \left(\delta^j_i F(\theta_i \otimes \ones_{0,1}) + (\epsilon^j_i + \delta_j^0 \delta^1_i) F(\theta_i \otimes \onesp_{0,1})\right)\right) \\
		&= \sum_{i \in I } \sum_{k \in I\setminus\{1\}} \left(\delta_j^0 \delta^1_i \delta^i_k + \delta^j_i \epsilon^i_k + \epsilon^j_i \delta^i_k\right) \thetabar^k = 0. 
		\quad\quad
		\text{(since $\partial_{2,1}\circ\partial_{2,1}=0$ and using (\ref{eq:F-v0-theta-infty}))}
	\end{align*}	
	For the second identity, note that:
	\begin{align*}
		\sum_{j \in J} F(\theta^i_j \otimes \theta^j) &= \delta^1_i F(\ones_{0,1} \otimes \ones_{1,2}) + \sum_{j \in I} \left(\delta^j_i F(\ones_{0,1} \otimes \theta^j) + \epsilon^j_i F(\onesp_{0,1} \otimes \theta^j)\right) \\
		&=\delta^1_i\onesp_{2,0}+\sum_{j\in I}\delta^j_i\delta^1_j\ones_{2,0}+
		\sum_{j,k\in I} \big(\delta^j_{i}\epsilon^k_j +\epsilon^j_{i}\delta^k_j\big) \theta_k
\quad\quad\quad\quad\quad
\text{(replacing (\ref{eq:F-theta-infty-v+}))}	\\
		&= (\delta^1_i + \delta^1_i)\ones_{2,0} = 0 
	\end{align*}
This completes the proof of the Lemma.
\end{proof}

\subsection{Main Exactness Theorem}

\begin{thm}\label{thm:exactness}
	For an exact rational tangle $R$ with characteristic link triple $(\Link_{0,1}= \Ucal, \Link_{1,2}, \Link_{2,0})$ and every corresponding skein triple $(\Knot_0, \Knot_1, \Knot_2)$, there is an exact triangle:
	\begin{center}
		\begin{tikzcd}[row sep=large, column sep=large]
			\HFKhat(\Knot_0)\otimes \frac{\HFKhat(\Link_{2,0})}{\Ker(\partial)} \arrow[rr,"\fmap"] 
			&& \HFKhat(\Knot_1)\otimes \Ker(\partial|_{\HFKhat(\Link_{1,2})}) \arrow[dl,"\gmap"] \\
			& \HFKhat(\Knot_2) \arrow[ul,"\hmap"] &
		\end{tikzcd}
	\end{center}
\end{thm}

\begin{proof}
	Using Lemmas~\ref{lem:1st-composition-is-trivial} and \ref{lem:2-3-compositions-are-trivial}, it suffices to verify exactness at each vertex:
	
	1. \textbf{Exactness at $\HFKhat(\Knot_2)$:} The element
	\begin{align*}
		\x_2 &= H((\partial \theta_1) \otimes \onesp_{0,1} \otimes \theta^0)  + \sum_{i,j \in I} H\left(\epsilon^j_i ((\partial \theta_i) \otimes \onesp_{0,1} \otimes \theta^j) + \delta^j_i ((\partial \theta_i) \otimes \ones_{0,1} \otimes \theta^j)\right)\\
		&=(H\circ\partial)\Big(\theta_1\otimes \onesp_{0,1}\otimes \ones_{1,2}+\sum_{i,j\in I}\big(\epsilon^j_i  \big(\theta_i\otimes\onesp_{0,1}\otimes \theta^j\big)+\delta^j_i\big(\theta_i\otimes\ones_{0,1}\otimes \theta^j\big)\big)\Big)
	\end{align*}
	equals $\ones_2$ modulo $\Hbb_{2,2}^\star$. For this, note that  since $\partial$ is trivial on $C_{2,2}$, we have $\partial \circ H=0$. Using Lemma~\ref{lem:vanifhing-of-F1} and (\ref{eq:compositios-law-3}) therein,  we may thus compute $\x_2$, as an element in $\Hbb_{2,2}/\Hbb_{2,2}^\star$  as follows:
	\begin{align*}
		\x_2&=F'\big(\theta_1\otimes F(\onesp_{0,1}\otimes \ones_{1,2}\big)\big)+\sum_{i,j\in I}F'\Big(  \theta_i\otimes\big(
		\epsilon^j_i  F\big(\onesp_{0,1}\otimes \theta^j\big)+\delta^j_i F\big(\ones_{0,1}\otimes \theta^j\big)\big)\Big)\\
		&\quad\quad+F'\big(F(\theta_1\otimes \onesp_{0,1})\otimes \ones_{1,2}\big)
		+\sum_{i,j\in I}
		F'\Big( \big(\epsilon^j_i  F\big(\theta_i\otimes\onesp_{0,1}\big)+\delta^j_i F\big(\theta_i\otimes\ones_{0,1}\big)\big) \otimes \theta^j\Big)\\
		&=
		F'\big(\onesp_{1,2}\otimes \ones_{1,2}\big)
		+\sum_{i,j\in I, k\in J} \big(\epsilon^j_i  \delta^i_k+\delta^j_i \epsilon^i_k\big) 
		F'\big(\theta^k \otimes \theta^j\big)
		+\sum_{i,j,k\in I}\big(\epsilon^j_i  \delta_j^k+\delta^j_i  \epsilon_j^k\big)
		F'\big(\theta_i\otimes \theta_k\big)=\ones_R.
	\end{align*}
	
	2. \textbf{Exactness at $\HFKhat(\Knot_0)$:} For each $i,i' \in I$, we need to show that
	\begin{align*}
		\y^i_{i'} &= \sum_{j \in J} (H \circ \partial)\left(\delta^j_i (\ones_{0,1} \otimes \theta^j \otimes \theta_{i'}) + \epsilon^j_i (\onesp_{0,1} \otimes \theta^j \otimes \theta_{i'})\right) = \delta^i_{i'} \cdot \ones_0.
	\end{align*}
	Since $\partial\circ H=0$, $d_{0,0}=0$ and $\ones_0$ is not in the image of $F:\Hbb_{0,2}\otimes\Hbb_{2,0}$, we have
	\begin{align*}
		\y^i_{i'}&=\sum_{j\in J}\delta^j_i \Big(
		F'\big(\ones_{0,1}\otimes F(\theta^j\otimes \theta_{i'})\big)+
		F'\big(F(\ones_{0,1}\otimes \theta^j)\otimes \theta_{i'}\big)\Big)\\
		&\quad\quad+\sum_{j\in J}\epsilon^j_i\Big(
		F'\big(\onesp_{0,1}\otimes F(\theta^j\otimes \theta_{i'})\big)+
		F'\big(F(\onesp_{0,1}\otimes \theta^j)\otimes \theta_{i'}\big)\Big)\\
		&=\Big(\epsilon^{i'}_{i}F'(\ones_{0,1}\otimes\ones_{0,1})+
		\delta^{i'}_iF'(\ones_{0,1}\otimes\onesp_{0,1})+\delta^1_i F'(\ones_{2,0}\otimes \theta_{i'})+\sum_{k\in I}\epsilon^k_iF'(\theta_k\otimes \theta_{i'})\Big)\\
		&\quad\quad+\sum_{j\in J}\epsilon^j_i\Big(
		\epsilon^{i'}_jF'(\onesp_{0,1}\otimes\ones_{0,1})+\delta^{i'}_jF'(\onesp_{0,1}\otimes\onesp_{0,1})+\delta^0_jF'(\ones_{2,0}\otimes \theta_{i'})+\sum_{k\in I}\delta^k_jF'(\theta_k\otimes \theta_{i'})\Big)\\
		&=\delta^i_{i'} \cdot \ones_0+\delta^1_i F'(\ones_{2,0}\otimes \theta_{i'})=\delta^i_{i'} \cdot \ones_0.
	\end{align*}
The last two  equalities follow since $F'$ vanishes on  $\onesp_{0,1}\otimes\ones_{0,1}$, $\ones_{2,0}\otimes \theta_k$ and    $\ones_{0,1}\otimes\ones_{0,1}-\onesp_{0,1}\otimes\onesp_{0,1}$.

	3. \textbf{Exactness at $\HFKhat(\Knot_1)$:} For each $j,j' \in J$, we need to show that
	\begin{align*}
		\z^j_{j'} &= \sum_{i \in I} (H \circ \partial)\left(\delta^{j'}_i (\theta^j \otimes \theta_i \otimes \ones_{0,1}) + \epsilon^{j'}_i (\theta^j \otimes \theta_i \otimes \onesp_{0,1})\right) = \delta^j_{j'} \cdot \ones_1
	\end{align*}
Since $\partial\circ H=0$, $d_{1,1}=0$ and $\ones_1$ is not in the image of $F:\Hbb_{1,2}\otimes\Hbb_{2,1}$, we have
	\begin{align*}
	\z^j_{j'}&=\sum_{i\in I}\delta_i^{j'}\Big(F'\big(F(\theta^j\otimes \theta_i)\otimes \ones_{0,1}\big)+F'\big(\theta^j\otimes F(\theta_i\otimes\ones_{0,1})\big)\Big)\\
	&\quad\quad+\sum_{i\in I}\epsilon_i^{j'}\Big(F'\big(F(\theta^j\otimes \theta_i)\otimes \onesp_{0,1}\big)+F'\big(\theta^j\otimes F(\theta_i\otimes\onesp_{0,1})\big)\Big)\\
	&=\sum_{i\in I}\delta_i^{j'}\Big(\epsilon^i_jF'(\ones_{0,1}\otimes\ones_{0,1})
	+\delta^i_jF'(\onesp_{0,1}\otimes\ones_{0,1})
	+\sum_{k\in I}\epsilon^i_kF'(\theta^j\otimes \theta^k)\Big)\\
	&\quad\quad+\sum_{i\in I}\epsilon_i^{j'}\Big(\epsilon^i_jF'(\ones_{0,1}\otimes\onesp_{0,1})
	+\delta^i_jF'(\onesp_{0,1}\otimes\onesp_{0,1})
	+\sum_{k\in I}\delta^i_kF'(\theta^j\otimes \theta^k)\Big)\\
	&=\epsilon^{j'}_jF'\big(\ones_{0,1}\otimes\ones_{0,1}
	+\onesp_{0,1}\otimes\onesp_{0,1}\big)
	+\delta^{j'}_j\cdot\ones_1
	=\delta^j_{j'} \cdot \ones_1.
\end{align*}	 	
For thelast two equalities, note that $F'(\ones_{0,1}\otimes\onesp_{0,1})=0$ and  $F'(\ones_{0,1}\otimes\ones_{0,1})=F'(\onesp_{0,1}\otimes\onesp_{0,1})$. 
\end{proof}

\section{A Family of Exact Rational Tangles}\label{sec:exact-RTs}
\begin{thm}
	The rational tangle corresponding to $2n+1$  positive half-twists is  exact.
\end{thm}	
\begin{figure}
	\def\svgwidth{0.9\textwidth}
	{\small{
			\begin{center}
				\input{Htriple.pdf_tex}
	\end{center}}}
	\caption{A Heegaard triple $(S^3,\beta_0,\beta_1,\beta_2,\ps,\qs)$.}
	\label{fig:Htriple}
\end{figure} 
\begin{proof}
	In this case, $L_{2,0}=T_{2,2n+1}$, $L_{1,2}=-T_{2,2n}$, $I=\{1,2,\ldots,2n\}$ and
\begin{align*}
\HFKhat(\Link_{2,0})=\langle \theta_i,\thetabar_i\ |\ i=0,1,\ldots,2n\rangle	
\end{align*}	
with $\partial\theta_0=0$, $\partial\thetabar_i=0$ for all $i\in\{0,1,\ldots,n\}$ and $\partial\theta_i=\thetabar_i$ for  $i\in I=\{1,\ldots,n\}$.  In fact, we may construct a Heegaard triple $(S^3,\beta_0,\beta_1,\beta_2,\ps,\qs)$ corresponding to $(\Link_{0,1},\Link_{1,2},\Link_{2,0})$, which is illustrated in Figure~\ref{fig:Htriple} for $n=2$. The general patern for arbitrary $n$ is obvious. The curves $\beta_0,\beta_1$ and $\beta_2$ are illustrated by blue, black and red color, respectively. Having fixed the labels for intersection points as in Figure~\ref{fig:Htriple}, for $i\in\{0,1,\ldots,2n\}$ we may set
\begin{align*}
\theta_{i}=
\begin{cases} x_{i/2}&\text{if $i$ is even}\\ y_{(i+1)/2}&\text{if $i$ is odd}\end{cases}
\quad\quad\quad\text{and}\quad\quad\quad
\thetabar_{i}=
\begin{cases} y_0&\text{if $i=0$}\\ y_i+x_{i-1}& \text{if $i>0$}\end{cases}.	
\end{align*}	
to obtain the aforementioned presentation of $(\HFKhat(\Link_{2,0}),\partial)$.
On the other hand, 
\begin{align*}
	\HFKhat(\Link_{2,1})=\langle \theta^i,\thetabar^i\ |\ i=0,1,\ldots,2n\rangle	
\end{align*}	
with $\partial^*\theta^1=0$, $\partial^*\thetabar^i=0$ for all $i\in I$ and $\partial^*\theta^i=\thetabar^i$ for  $i\in \{2,3,\ldots,n\}$.  The dual differential $\partial=\partial_{1,2}$ is given by $\partial\thetabar^1=0$, $\partial\thetabar^i=\theta^i$ for $i\in\{2,3,\ldots,n\}$ and $\partial\theta^i=0$ for $i\in I$. With reference to Figure~\ref{fig:Htriple}, it suffices that for $i\in\{1,\ldots,n\}$ we set
\begin{align*}
\theta^{i}=
\begin{cases} x^{i/2}&\text{if $i$ is even}\\ y^{(i+1)/2}&\text{if $i$ is odd}\end{cases}
\quad\quad\quad\text{and}\quad\quad\quad
\thetabar^{i}=
\begin{cases} y^1&\text{if $i=1$}\\ y_i+x_{i-1}& \text{if $i>1$}\end{cases}.		
\end{align*}	

The triangle map $F(\cdot\otimes\ones_{0,1})$ is then given by 
\begin{align*}
&{f}=F(\cdot\otimes\ones_{0,1}):\HFKhat(\Link_{2,0})\ra \HFKhat(\Link_{2,1}),\\
&{f}(\theta_i)=\theta^i\quad\text{and}\quad {f}(\thetabar_i)=\thetabar^i\quad\forall\ i\in I
\quad\quad\quad\text{and}\quad\quad\quad
{f}(\theta_0)={f}(\thetabar_0)=0.	
\end{align*}
Note that $\ones_{1,2}=\thetabar^1$ and $\onesp_{1,2}=\theta^1$, while $\ones_{2,0}=\theta_0$ and $\onesp_{2,0}=\thetabar_0$. The above conclusion follows from a direct computation using the simple Heegaard triple of Figure~\ref{fig:Htriple}.
In fact, there are small triangles $\Delta_i$ connecting the top generator $\ones=\ones_{0,1}$, $x_i$ and $x^i$ for $i=1,\ldots,2n$, small triangles $\Delta'_i$ connecting $\ones,y_i$ and $y^i$ for $i=2,\ldots,n$, and a triangle $\Delta'_1$ which connects $\ones, x_0$ and $y^1$. One can check that these are the only possible contributing triangles for the holomorphic triangle map ${f}$. This implies the above claim about the evaluation of ${f}$ on $\theta_i$ and $\thetabar_i$.\\
 
In particular, it follows from the above observation/computation that the induced map
\begin{align*}
	f:\frac{\HFKhat(\Link_{2,0})}{\Ker(\partial_{2,0})}=\langle \theta_i\ |\ i\in I\rangle 
		\ra 
		\frac{\Ker(\partial_{1,2})}{\langle\ones_{1,2}\rangle}=
		\langle \theta^i\ |\ i\in I\rangle
\end{align*}	
is an isomorphism (given by sending $\theta_i$ to $\theta^i$ for every $i\in I$). This completes the proof of the theorem once we use Theorem~\ref{thm:exactness}.
\end{proof}	
\bibliographystyle{hamsalpha}

\end{document}